\documentclass[a4paper, 12pt]{amsart}
\usepackage[T1]{fontenc}
\usepackage[utf8]{inputenc}
\usepackage[letterpaper,top=2cm,bottom=2cm,left=3cm,right=3cm,marginparwidth=1.75cm]{geometry}
\usepackage{amsmath, amssymb, graphicx}
\usepackage[colorlinks=true, allcolors=blue]{hyperref}
\newtheorem{thm}{Theorem}[section]
\newtheorem{prop}[thm]{Proposition}
\newtheorem{lemma}[thm]{Lemma}
\newtheorem{cor}[thm]{Corollary}

\DeclareMathOperator{\supp}{\mathrm{supp}}

\begin{document}

\title[Strong law of large numbers for generalized operator means II]{Strong law of large numbers for generalized operator means II}
\author[Zolt\'an L\'eka] {Zolt\'an L\'eka}
\address{Institute of Applied Mathematics, \'Obuda University, John von Neumann Faculty of Informatics, 1034 Budapest Bécsi út 96, Hungary.}
\email{leka.zoltan@uni-obuda.hu}
\author[Mikl\'os P\'alfia]{Mikl\'os P\'alfia}
\address{Department of Mathematics, Corvinus University of Budapest, 1093 Budapest F\H ov\'am t\'er 8, Hungary and Bolyai Institute, Interdisciplinary Excellence Centre, University of Szeged, H-6720 Szeged, Hungary.}
\email{miklos.palfia@uni-corvinus.hu}


\date{\today}

\begin{abstract}
   In this paper we continue the investigation of generalizations of the strong law of large numbers to operator means which has been initiated in the earlier paper "Strong law of large numbers for generalized operator means",  Adv. Math. 457 (2024), 109933 of the authors. In particular, we clarify an inaccuracy in the statement and proof of the strong law for means associated with uniformly exponentially contracting flows in Thompson metric spaces. We then significantly improve the variant of this strong law given in the mentioned paper for operator means by relaxing the restrictive required moment condition to $L\log^+(L)$ by proving almost sure convergence of implicit stochastic approximations for monotone operators under this condition.
\end{abstract}

\maketitle

\section{Introduction}


The first goal here is to address two technical issues concerning \cite[Theorem 5.2]{lekapalfia}. First, we make explicit an omitted local $L^1$-boundedness type of assumption and modify the proof accordingly. Second, we address a separate technical gap in the proof: certain second order error terms in the stochastic analogue of \cite[(20)]{lekapalfia} are not controlled under the stated assumptions, and may, in fact, be unbounded. It follows that one cannot directly apply the "Nodice" theorem \cite[Theorem 4.2]{lekapalfia} as claimed at the beginning of the last paragraph of the proof of \cite[Theorem 5.2]{lekapalfia}. This isssue, however, does not affect the subsequent application of the results in \cite{lekapalfia} including Section 6, because the local $L^1$-boundedness condition introduced in the present paper is automatically satisfied in that setting, as verified in Section 4 below, see Corollary~\ref{C:UniformInt}.

Resolving these technical issues is only one aspect of our contribution. More importantly, our second goal is to substantially strengthen the previous convergence results by proving them under considerably weaker moment assumptions. In the final section, we strengthen the strong law for operator means in \cite[Theorem 6.9]{lekapalfia} and impose a single integrability condition in the form of an $L\log^+(L)$ condition which lies strictly between the theoretically minimal $L^1$ and the stronger, but usually proof-wise more accessible, $L^2$-condition. Recently, \cite{nguyen2026} established almost sure convergence results for explicit stochastic approximation schemes in $\mathbb{R}^d$ under $L^{1+\varepsilon}$-type conditions. Another almost sure convergence result is \cite{bianchi} for $L^2$-integrable implicit stochastic approximations in Hilbert spaces. The proof in \cite{bianchi} relies on the Robbins-Siegmund almost-supermartingale convergence theorem \cite{robbinssiegmund}, a the standard tool in stochastic approximation theory.  In our setting, we avoid supermartingale arguments and prove almost sure convergence results more directly and effectively under the weaker $L\log^+(L)$ tail condition. Having established this and using the order-preserving property of resolvent maps, our treatment reduces the problem of the strong law for generalized operator means to an implicit stochastic approximation problem on the real line. The proximal Robbins--Monro method \cite{toulis2021} provides a general framework for such implicit stochastic approximation schemes. Then we apply our result to establish almost sure boundedness of generalized resolvent iterates associated with the mean ODE. We reduce the stochastic convergence analysis to bounded metric balls, where the generated flow is invariant and uniformly exponentially contractive, and adapt ideas from optimal transport theory and the proof of \cite[Theorem 6.9]{lekapalfia} to derive the convergence. Without the flow-invariance argument, our estimates would prove to be ineffective to yield the strong law, since the exponential contraction coefficient of the mean ODE could be $0$ on the whole cone $\mathbb{P}$.

\section{Cost functions in optimal transport plans}

  All Borel probability measures in $\mathcal{P}(\mathbb{P})$ representing random variables are assumed to be fully supported, where the positive definite cone $\mathbb{P}$ of invertible operators over a Hilbert space $\mathcal{H}$ is equipped with the Thompson metric $d_\infty(a,b)=\|\log(a^{-1/2}ba^{-1/2})\|$. Recall that the support of a Borel probability measure is always separable \cite{Dudley}. Let us consider the function
  $$ \phi(x) = \int_\mathbb{P} \varphi(x,y) \: d\mu(y)$$
  where $\mu \in \mathcal{P}(\mathbb{P})$ for some measurable, locally $\|\cdot\|$-Lipschitzian function family $\varphi(\cdot,y)$ for any $y\in\mathbb{P}.$ The corresponding mean ODE is given as
\begin{equation}\label{eq:IntroODE}
\dot{x}(t) = \int_\mathbb{P} \varphi(x(t),y) \: d\mu(y)=\mathbb{E}\varphi(x(t),Y),
\end{equation}  
 where $Y$ is a $\mathbb{P}$-valued random variable with law $\mu$. The generated flow  on $\mathbb{P}$ is (globally) exponentially contractive if there exists an $\alpha > 0$ such that
   $$ d_\infty(x(t),y(t)) \leq e^{-\alpha t} d_\infty(x(0),y(0)) \quad t \geq 0$$
   for any solution curves $x,y:[0,\infty)\mapsto\mathbb{P}$ of \eqref{eq:IntroODE}.
This exponential contraction rate of the mean ODE is denoted by $\alpha(\varphi, \mu)$ for brevity or just by $\alpha$ if it is clear from the context.

 For any $\mathbb{P}$-valued random variable $Y$ with law $\mu,$ recall that $J_{\lambda}(Y, z) = J_{\lambda}^{\delta_Y}(z)$ for $\lambda>0$ denotes the unique solution to the resolvent equation
 $$ \lambda \varphi(x,Y)  + \log_x z=0$$
 in $x\in\mathbb{P}$. The existence and uniqueness of the solution follow from the fact that it is the unique equilibrium of the differential equation
   $$ \dot{x}(t)=\lambda \varphi(x(t),Y)  + \log_{x(t)} z$$
 which generates a globally exponentially contractive flow. Indeed, we have that the $\log$ function defines an order-preserving dynamical system with exponential contraction rate $1$, see \cite[Proposition 2.5]{lekapalfia}. Therefore, the contraction rate of the above flow is at least $1 + \lambda \alpha(\varphi(\cdot , Y)) > 0$, given $\varphi(\cdot, y)$ is non-expansive.

 We also define the generalized resolvent $J_\lambda^\mu(z)$ to the mean ODE as the unique solution of
  $$ \lambda \int_\mathbb{P} \varphi(x,y) \: d\mu(y) + \log_x z=0.$$

Let us introduce the continuous cost function $\displaystyle c_x$ on $\mathbb{P} \times \mathbb{P}$ defined as
\begin{equation*}
 c_x(y,z) = \| x^{-1/2}(\varphi(x,y) - \varphi(x,z))x^{-1/2}\|.
\end{equation*}
Then
  $$ c_x(y,z) \leq c_x(y,v) + c_x(v,z)$$ holds. Hence, $c_x$ is a pseudometric, but not necessarily a metric on $\mathbb{P}.$
  Since $c_x$ is continuous, the optimal transport cost functional between the
  Borel probability measures $\displaystyle \mu$ and $\displaystyle \nu$ on $\mathbb{P}$ is well-defined
\begin{equation*}
C_x(\mu, \nu) = \inf_{\pi  \in \Pi(\mu, \nu)} \mathbb{E}_\pi c_x(y,z) = \min_{\pi \in \Pi(\mu, \nu)} \int_{\mathbb{P} \times \mathbb{P}}
c_x(y,z) \: d\pi(y,z),
\end{equation*}
where $\Pi(\mu, \nu)$ denotes the set of all couplings between $\mu$ and $\nu.$

We now have the following form of a general resolvent estimate, see Corollary 3.9 in \cite{lekapalfia}:
\begin{cor}\label{C:fullresolventineq}
Let $\varphi(x,y)$ be a locally Lipschitz function on $\mathbb{P} \times \mathbb{P}.$
Let $\mu$ and $\nu$ be Borel probability measures on $\mathbb{P}$, let $p, q \in \mathbb{P}$ and let $\lambda>0$. Then, for the resolvents of the corresponding mean ODEs, we have
   \begin{equation}\label{ineq:resolvent_cost} d_\infty(J^\mu_\lambda(p), J_\lambda^\nu(q)) \leq {1 \over 1 +  \lambda \alpha(\varphi,\mu)} d_\infty(p,q)
 + {\lambda \over 1 +  \lambda \alpha(\varphi, \mu)}C_{J_\lambda^\nu(p)}(\mu, \nu).
\end{equation}
\end{cor}
\begin{proof}
  Fix an $x \in \mathbb{P}.$ Using some linear functional $\omega$ with $\|\omega\| = 1,$ note that
\begin{align*}
&\left\|\int_\mathbb{P} x^{-1/2}\varphi(x,y)x^{-1/2} \: d\mu(y) - \int_\mathbb{P} x^{-1/2}\varphi(x,z)x^{-1/2} \: d\nu(z) \right\| \\
&\quad = \omega \left( \int_\mathbb{P} x^{-1/2}\varphi(x,y)x^{-1/2} \: d\mu(y) - \int_\mathbb{P} x^{-1/2}\varphi(x,z)x^{-1/2} \: d\nu(z) \right) \\
&\quad = \int_\mathbb{P} \omega (x^{-1/2}\varphi(x,y)x^{-1/2}) \: d\mu(y) - \int_\mathbb{P} \omega (x^{-1/2}\varphi(x,y)x^{-1/2}) \: d\nu(z).
\end{align*}
Now, \begin{align*}
\omega(x^{-1/2}\varphi(x,y)x^{-1/2})& - \omega(x^{-1/2}\varphi(x,z)x^{-1/2})\\
&=  \omega(x^{-1/2}\varphi(x,y)x^{-1/2} - x^{-1/2}\varphi(x,z)x^{-1/2}) \\
&\leq \|x^{-1/2}(\varphi(x,y) - \varphi(x,z))x^{-1/2}\| \\
& = c_x(y,z).
\end{align*} From
the Kantorovich duality theorem, we obtain
\begin{align*}
&\left\| \int_\mathbb{P} x^{-1/2}\varphi(x,y)x^{-1/2} \: d\mu(y) - \int_\mathbb{P} x^{-1/2}\varphi(x,z)x^{-1/2} \: d\nu(z) \right\| \\
&\hspace{2cm} \leq \sup_{\substack{
(\phi, \psi) \in L^1(\mu) \times L^1(\nu) \\
\phi - \psi \leq c_x }}
\left\{ \int_\mathbb{P} \phi(y) \: d\mu(y) - \int_\mathbb{P} \psi(z)  \: d\nu(z) \right\} \\
 &\hspace{2cm} = \min_{\pi \in \Pi(\mu, \nu)} \int_{\mathbb{P} \times \mathbb{P}}
c_x(y,z) \: d\pi(y,z) = C_x(\mu,  \nu).
\end{align*}
Choosing $x = J_\lambda^\nu(p),$ the result follows from Proposition 3.4 and Corollary 3.8 in \cite{lekapalfia}.
\end{proof}

Assume that for each $y$ the vector field $\varphi(\cdot, y)$ defines a globally exponentially contractive flow on $\mathbb{P}$. Let $\Lambda_{\mu}$ denote the unique solution to the mean equation
\begin{equation*}
 \int_\mathbb{P} \varphi(x,y) \: d\mu(y) = 0
\end{equation*}
 for $x$, cf. \cite[Proposition 3.1]{lekapalfia}. Notice that
$\displaystyle \Lambda_{\mu}$\ is the fixed point of the non-linear differential equation
\begin{equation*}
\dot{x}(t) = \int_\mathbb{P} \varphi(x(t),y) \: d\mu(y).
\end{equation*}
Let $\displaystyle \nu$ be another Borel probability measure on $\mathbb{P}$. Applying inequality in \cite[Proposition 3.6]{lekapalfia} to the differential equation above, we immediately obtain
\begin{align*}
d_\infty(\Lambda_{\mu},\Lambda_{\nu}) &\leq {1 \over \alpha(\varphi, \mu)} \left\| \int_\mathbb{P} \varphi(\Lambda_{\nu},y) \: d \mu(y) \right\| \\
&=  {1 \over \alpha(\varphi, \mu)} \left\| \int_\mathbb{P} \varphi(\Lambda_{\nu},y) \: d \mu(y) - \int_\mathbb{P} \varphi(\Lambda_{\nu},z) \: d \nu(z) \right\| \\
&= {1 \over \alpha(\varphi, \mu)}  \left(\int_\mathbb{P} \omega(\varphi(\Lambda_{\nu},y)) \: d \mu(y) - \int_\mathbb{P} \omega(\varphi(\Lambda_{\nu},z)) \: d \nu(z) \right)
\end{align*}
for some linear functional $\omega$ with norm $1$.

Similarly, the Kantorovich duality theorem implies
  \begin{align*} \int_\mathbb{P} \omega(\varphi(\Lambda_{\nu},y)) \: d \mu(y) - \int_\mathbb{P} \omega(\varphi(\Lambda_{\nu},z)) \: d \nu(z) &\leq \min_{\pi \in \Pi(\mu, \nu)} \int_{\mathbb{P} \times \mathbb{P}}
c_{\Lambda_{\nu}}(y,z) \: d\pi(y,z) \\ &=
 C_{\Lambda_{\nu}}(\mu, \nu).
 \end{align*}

Hence, we obtain
\begin{cor}[Distance between fixed points]\label{C:distancefixedpoints}
  \begin{equation*}
d_\infty(\Lambda_{\mu},\Lambda_{\nu}) \leq {1 \over \alpha(\varphi, \mu)} C_{\Lambda_{\nu}}(\mu, \nu).
\end{equation*}
\end{cor}

 \bigskip

\section{Stochastic resolvent iterations}

 We now present the main result of this section. It clarifies the additional local $L^1$-domination condition, which in particular implies a uniform integrability-type condition on balls, used in \cite[Theorem 5.2]{lekapalfia} and substantially expands and elucidates parts of the proof. The new ingredients in our argument include tools from optimal transport theory, most notably Kantorovich duality \cite{villani_oldnew}.

\begin{thm}\label{T:mainStochasticConvergence}
    Assume that \(\varphi: \mathbb{P} \times \mathbb{P} \rightarrow \mathbb{S}\) is Lipschitz continuous on any bounded \(d_\infty\)-metric ball. Furthermore, for each \(x \in \mathbb{P}\), there exists an \(\varepsilon > 0\) such that the family of functions
    $$ \mathcal{F}_{x, \varepsilon} = \left\{ \left\| s^{-1/2} \varphi(s, y) s^{-1/2} \right\| \colon d_\infty(s,x) < \varepsilon \right\} $$
admits a common $L^1$-dominating function. In addition, for every \(y \in \mathbb{P}\), we have \(\alpha(\varphi(\: \cdot \:, y)) \geq \alpha > 0\). Let \(Y_1, Y_2, \ldots\) be a sequence of i.i.d. random variables with law \(\mu \in \mathcal{P}(\mathbb{P})\). Then, for any initial point \(S_0 \in \mathbb{P}\), the stochastic resolvent iteration defined by
$$
S_{n+1} = J_{\frac{1}{n+1}}(Y_{n+1}, S_n)
$$
almost surely converges to \(\Lambda_\varphi(\mu)\) in the Thompson metric.
\end{thm}

\begin{proof}

 We divide the argument into several steps. First, as usual, fix a sufficiently small $\varepsilon > 0$.

\bigskip

{\bf Step 1.} From Varadarajan’s empirical large number theorem, the empirical averages
$\displaystyle \mu_k := {1 \over k} \sum_{i=1}^k \delta_{Y_i}$\ converge weakly to $\displaystyle \mu$\ a.s. if
$\displaystyle k \to \infty.$\ \ Let us now define the cost function
$$
\bar{c}(y,z) = \mbox{sup} \{ c_x(y,z) \colon d_\infty(x,\Lambda_{ \mu}) < \varepsilon)\}.
$$
Since each $c_x$ is continuous on $\mathbb{P} \times \mathbb{P}$, the supremum $\displaystyle \bar{c}$ is lower semi-continuous on $\mathbb{P} \times \mathbb{P}$. Moreover, by the local $L^1$-domination condition,  for every fixed
$\displaystyle z \in \mathbb{P},$ the function $\bar{c}(\cdot, z)$ belongs to $\displaystyle L^1(\mu) $. Hence, the optimal transport cost between the Borel probability measures $\displaystyle \mu$\ and
$\displaystyle \mu_k$\ is well-defined
\begin{equation*}
\bar{C}(\mu, \mu_k) = \min_{\pi \in \Pi(\mu, \mu_k)} \mathbb{E}_\pi \bar{c}(y,z).
\end{equation*}
Applying the Strong Law of Large Numbers to $\displaystyle \bar{c}(\cdot, z_0) \in L^1(\mu)$, we also have that
$$\int_\mathbb{P}  \bar{c}(y, z_0) \: d\mu_k(y) \to \int_\mathbb{P}  \bar{c}(y, z_0) \: d \mu(y) \mbox{ a.s. } k \rightarrow \infty.$$
In addition, since $\displaystyle \mu_k$ converges to $\displaystyle \mu$ weakly, the stability theorem of optimal
transport plans with a standard truncation approach, see \cite[p. 114]{villani_oldnew}, implies the convergence
\begin{align*}
\min_{\pi \in \Pi(\mu, \mu_k)} \mathbb{E}_\pi \bar{c}(y,z) \to 0, \quad k \to \infty.
\end{align*}
We have $C_{\Lambda_\mu}(\mu, \mu_k) \leq \bar{C}(\mu, \mu_k),$ and then $d_\infty (\Lambda_{ \mu_k}, \Lambda_{ \mu}) \to 0 \; (k \to \infty)$ a.s. also holds, see Corollary~\ref{C:distancefixedpoints}. Since
$\displaystyle \bar{c}$\ is lower semi-continuous, the map $\displaystyle (\mu, \nu) \mapsto \bar{c}(\mu, \nu) $ is
l.s.c. in the weak topology on the product space of probability measures. The infimum of any family of l.s.c.
functions is Borel measurable, and we obtain for any fixed $k$ that $\bar{C}(\mu,
\mu_k)$\ is measurable with respect to the Borel $\sigma$-algebra on $\mathbb{P}^{\otimes k}$ (and
$\mathbb{P}^\mathbb{N}$). We claim that
$$\displaystyle \mathbb{E}\bar{C}(\mu, \mu_k) \rightarrow 0, \quad  k \rightarrow \infty.$$
Indeed, we use a standard decomposition trick for the cost function $\displaystyle \bar{c}$ that satisfies the triangle inequality. Then, we obtain (see \cite[p. 114]{villani_oldnew})
\begin{align*}
\int_{\mathbb{P} \times \mathbb{P}} \bar{c}(x,y) d\pi_k(x,y) &\leq \int_{\{\bar{c}(x,y) < R\}} \bar{c}(x,y) d\pi_k(x,y)  \\
& \quad + 2\int_{\{\bar{c}(x,z_0) > R/2\}} \bar{c}(x,z_0) d\mu(x) + 2\int_{\{\bar{c}(z_0,y) > R\}} \bar{c}(z_0,y) d\mu_k(y).
\end{align*}
Hence
\begin{equation*}
\mathbb{E}\bar{C}(\mu, \mu_k) \leq \mathbb{E} \int_{\{\bar{c}(x,y) < R\}} \bar{c}(x,y) d\pi_k(x,y) +
4\int_{\{\bar{c}(x,z_0) > R/2\}} \bar{c}(x,z_0) d\mu(x).
\end{equation*}
Since $\displaystyle \mu_k \to \mu$\ weakly, from the stability of optimal transport plans \cite[Theorem 5.20]{villani_oldnew}, we have $\pi_k \to \pi$ weakly, where $\pi$\ is the trivial coupling $\displaystyle (Y, Y)$\ and
$\displaystyle Y \sim \mu$. Hence, Lebesgue’s dominated theorem implies that the expected value
of the right-hand side tends to $\displaystyle 0$. From the integrability of the cost functional $\displaystyle
\bar{c}( \cdot, z_0)$, the second integral is arbitrarily small if $\displaystyle R$\ is large enough. We conclude that
$\displaystyle \mathbb{E}\bar{C}(\mu, \mu_k) \rightarrow 0$ as $\displaystyle k \rightarrow \infty.$

\bigskip

\textbf{Step 2. }Now, fix $k$ such that $d_\infty(\Lambda_{\mu_k},\Lambda_{\mu}) < \varepsilon/2$\ a.s.,
\ $\displaystyle \bar{C}(\mu, \mu_k) < \varepsilon$\ a.s. and $\displaystyle \mathbb{E}\bar{C}(\mu, \mu_k) <
\varepsilon$\ are all fulfilled. Let $\displaystyle \pi_k$\ be the optimal transport plan defined by the optimal cost
$\displaystyle \bar{C}(\mu, \mu_k)$. Given the random variable $Y$\ with distribution $\mu$, the transport plan $\pi_k$\ also defines the coupling $(Y, Z)$, where $\displaystyle
Z$\ has distribution $\displaystyle \mu_k$. Thus, the sequence of i.i.d. random variables $Y_1 ,
Y_2 , \hdots $ with law $\mu$ generates a sequence of i.i.d. random variables $\displaystyle Z_1 , Z_2 , . . .$ with law $\displaystyle \mu_k .$ Let us introduce the stochastic resolvent sequence
\begin{equation*}
  {1 \over n+1} \varphi(\hat{S}_{n+1}, Z_{n+1}) + \log_{\hat{S}_{n+1}} \hat{S}_n  = 0, \quad n=1,2, \hdots .
\end{equation*}

\bigskip

\textbf{Step 3.} We prove that
\begin{equation*}
\limsup_{n \to \infty} d_\infty(\hat{S}_{n}, \Lambda_{ \mu_k}) \leq {\varepsilon \over \alpha}.
\end{equation*}
Note that the variables $Z_{n+1}$-s are $\mathbb{P}$-valued i.i.d.-s with finitely many values. From the resolvent estimate Corollary~\ref{C:fullresolventineq}, we obtain that the sequence
$\displaystyle \hat{S}_n$ is bounded almost surely. Indeed, with $\lambda = 1/(n+1)$,
\begin{align*}
d_\infty(\Lambda_{ \mu_k},\hat{S}_{n+1}) &= d_\infty(\Lambda_{ \mu_k}, J_{1/n}^{\delta_{Z_{n+1}}}(\hat{S}_{n+1}) \\ &\leq {n+1 \over n+1 + \alpha } d_\infty(\Lambda_{ \mu_k}, \hat{S}_n) + {\alpha \over
n+1+\alpha}{\|Z_{n+1}^{-1/2} \varphi(\Lambda_{\mu_k}, Z_{n+1}) Z_{n+1}^{-1/2} \| \over \alpha}.
\end{align*}
Since the variable $\|\varphi(\Lambda_{\mu_k}, Z)\|$\ is uniformly bounded, a simple induction shows that $\hat{S_n}$\ is
bounded almost surely.

Following the calculations in the no-dice theorem \cite[Theorem 4.2]{lekapalfia}, summing the resolvent equations
$$ {1 \over nk + i} \varphi(\hat{S}_{nk+i-1}, Z_{nk+i-1}) + \log_{ \hat{S}_{nk+i}}  \hat{S}_{nk+i-1}  = 0\quad   i=1, \hdots, k$$
yields the formula
\begin{equation}\label{eq:nodiceResolventWithError}
      {1 \over {(n+1)k}} \sum_{i=1}^{k} \varphi(\hat{S}_{nk+k},Z_{nk+i-1}) = -\log_{\hat{S}_{nk+k}} \hat{S}_{nk}   + O\left({1 \over (n+1)^2}\right).
      \end{equation}

Let us introduce an auxiliary random function: $$\displaystyle \psi_n(x) =   {1 \over k }\sum_{i=1}^k \varphi(x, Z_{nk+i-1})
=  \mathbb{E}_{\mu_{k,n}} \varphi(x, W),  $$ for each $\displaystyle n$,  where $\displaystyle \mu_{k,n}$ denotes the random atomic measure supported on $\displaystyle Z_{nk+i-1}$ with probability $\displaystyle 1/k$ (and $W \sim \mu_{k,n}$). We also consider the random resolvent iteration $S_n^r$ defined by
$$ {1 \over n+1} \psi_n(S_{n+1}^r) + \log_{S_{n+1}^r} S_n^r = 0, \quad n = 1,2, \hdots.$$
For each $n$, note that $S_{n+1}^r$ is the only fixed point of the differential equation
$$ \dot{x}(t) = {1 \over n+1} \psi_n(x(t)) + \log_{x(t)}S_n^r ,$$
since the generated flow is exponentially contractive on $\mathbb{P}.$ Indeed, by \cite[Corollary 2.2]{lekapalfia}, recall that $$\alpha \left({1 \over n+1} \psi_n(x(t)) + \log_{x(t)}S_n^r \right) \geq \alpha \left({1 \over n+1} \psi_n(x(t))\right) + 1.$$

From the contraction property \cite[Theorem 2.4]{lekapalfia} and \eqref{eq:nodiceResolventWithError}, we have
\begin{align*}
  d_\infty(S_{n+1}^r,\hat{S}_{nk+k}) &\leq {1 \over 1 + \alpha((n+1)^{-1}\psi_n)} \left\|\hat{S}_{nk+k}^{-1/2}\left( \frac{1}{n+1} \psi_n(\hat{S}_{nk+k}) + \log_{\hat{S}_{nk+k}} S_n^r \right)\hat{S}_{nk+k}^{-1/2} \right\| \\
  &= {1 \over 1 + \alpha((n+1)^{-1}\psi_n)} \| \hat{S}_{nk+k}^{-1/2}( - \log_{\hat{S}_{nk+k}} \hat{S}_{nk} + \log_{\hat{S}_{nk+k}} S_n^r)\hat{S}_{nk+k}^{-1/2}\| \\
  &\quad+ O((n+1)^{-2}) \\
  &\leq {1 \over 1 + \alpha / (n+1)} d_\infty( S_n^r, \hat{S}_{nk}) + O((n+1)^{-2}),
\end{align*}
where we used the exponential metric increasing inequality for Thompson metric and its invariance in the last step.
Thus, $d_\infty(S_n^r , \hat{S}_{nk+k}) \to 0$ as $n \to \infty.$

We also introduce the deterministic sequence
\begin{equation*}
   {1 \over n+1} \mathbb{E}_\mu \varphi(S^d_{n+1}, Y) + \log_{S^d_{n+1}} S^d_n  = 0, \quad n = 1,2, \hdots.
\end{equation*}
Recall that the algorithm converges in $d_\infty$; that is, $\displaystyle S_{n}^d \to \Lambda_{\mu}$, $\displaystyle n \to \infty,$ see \cite[Theorem 4.1]{lekapalfia}. Hence,
for any large $\displaystyle n,$ applying the resolvent inequality in Corollary~\ref{C:fullresolventineq} for the mean ODE $\dot{x} = \mathbb{E}_\mu(x, Y)$ and the empirical mean ODE $\dot{x} = \mathbb{E}_{\mu_{k,n}}(x, W),$ we obtain
\begin{align*}
d_\infty(S_{n+1}^d,S_{n+1}^r) &= d_\infty(J^\mu_{1/(n+1)}(S_n^d), J^{\mu_{k,n}}_{1/(n+1)}(S_n^r)) \\
 &\leq {1 \over 1 + \alpha(\varphi, \mu_{k,n})(n+1)^{-1}} d_\infty(S_{n}^d, S_{n}^r) \\
 &\quad+ {(n+1)^{-1} \over 1 + \alpha(\varphi, \mu_{k,n})(n+1)^{-1}}
C_{S_{n+1}^d}(\mu, \mu_{k,n}) \\
&\leq  {n+1  \over n+1 + \alpha} d_\infty(S_{n}^d , S_{n}^r) + {\alpha \over n+1 + \alpha} {\bar{C}(\mu, \mu_{k,n}) \over \alpha}.
\end{align*}
Note that $\bar{C}(\mu, \mu_{k,1}), \bar{C}(\mu, \mu_{k,2}), \hdots$\ are i.i.d. random variables defined
on $\mathbb{P}^{\otimes k}.$ Since the optimal transport cost is convex in both variables, we obtain
\begin{equation*}
\int_{\mathbb{P}^{\otimes k}} \bar{C}\left(\mu, {1 \over k}\sum_{i=1}^k Y_{nk+i-1} \right) d\mu^{\otimes k} \leq \int_{\mathbb{P}}
\bar{C}(\mu, \delta_{Y_1}) d\mu  = \int_{\mathbb{P}} \bar{c}(y, z) d\mu(z) < \infty;
\end{equation*}
that is, $\displaystyle \bar{C}(\mu, \mu_{k,n}) \in L^1(\mathbb{P}^{\otimes k}, \mu^{\otimes k})$. Hence, from the SLLN for weighted averages (see \cite[Theorem 1]{Etemadi}), we have
$$\displaystyle \limsup_{n \rightarrow \infty} d_\infty(S_{n+1}^d , S_{n+1}^r) \leq {1 \over  \alpha}\mathbb{E}\bar{C}(\mu,
\mu_k) \leq {\varepsilon \over \alpha}.$$
Since
\begin{equation*}
 d_\infty(\hat{S}_{nk},\Lambda_{ \mu_k}) \leq d_\infty(\hat{S}_{nk}, S_n^r) + d_\infty(S_n^r,S_n^d) + d_\infty(S_n^d , \Lambda_{ \mu_k}),
\end{equation*}
we find that
\begin{equation*}
\limsup_{n \to \infty} d_\infty(\hat{S}_{nk} , \Lambda_{ \mu_k}) \leq {\varepsilon \over \alpha}.
\end{equation*}
 From $d_\infty(\hat{S}_n,\hat{S}_{n+1}) \leq Kn^{-1}$, it follows
\begin{equation*}
\limsup_{n \to \infty} d_\infty(\hat{S}_{n},\Lambda_{ \mu_k}) \leq {\varepsilon \over \alpha}.
\end{equation*}

\bigskip

\textbf{Step 4. }We claim that
$$\displaystyle \limsup_{n \to \infty} d_\infty(\hat{S}_n,S_n) < { \varepsilon \over \alpha} \text{a.s.} $$
Since $d_\infty(\hat{S}_n,\Lambda_{\mu_k}) < \varepsilon/\alpha$, for any large $\displaystyle n$, and
$d_\infty(\Lambda_{\mu_k}, \Lambda_{\mu}) < \varepsilon/2$, we also obtain that $\displaystyle C_{\hat{S}_{n+1}}(\mu, \mu_k)
\leq \bar{C}(\mu, \mu_k) $ a.s. Moreover, for any large $\displaystyle n,$  by the general resolvent contraction inequality in Corollary~\ref{C:fullresolventineq}
\begin{align*}
d_\infty&(\hat{S}_{n+1},S_{n+1}) = d_\infty(J_{1/(n+1)}^{\delta_{Z_{n+1}}}{\hat{S}_n},
J_{1/(n+1)}^{\delta_{Y_{n+1}}}{S_n}) \\
&\leq {1 \over 1 + \alpha(\varphi, \delta_{Y_{n+1}}) (n+1)^{-1}} d_\infty(\hat{S}_n,  S_n) + {(n+1)^{-1} \over
1 + \alpha(\varphi, \delta_{Y_{n+1}}) (n+1)^{-1}}  c_{\hat{S}_{n+1}}(  Y_{n+1}, Z_{n+1}) \\
& \leq {1 \over 1 + \alpha(\varphi, \delta_{Y_{n+1}}) (n+1)^{-1}}
d_\infty(\hat{S}_n,S_n) + {(n+1)^{-1} \over 1 + \alpha(\varphi, \delta_{Y_{n+1}}) (n+1)^{-1}}  \bar{c}(  Y_{n+1}, Z_{n+1}) \mbox{ a.s.}
\end{align*}
From Etemadi's weighted SLLN (see \cite[Theorem 1]{Etemadi}) to the i.i.d. couples $(Y_n, Z_n) \sim \pi_k$, we have
$$ \limsup_{n \rightarrow \infty} d_\infty(\hat{S}_{n+1}, S_{n+1}) \leq  { \mathbb{E}_{\pi_k}\bar{c}(\mu, \mu_k) \over \alpha}  =  {\bar{C}(  \mu, \mu_k) \over \alpha}
\mbox{ a.s. } $$
Combining these inequalities, we obtain that, for every sufficiently large $\displaystyle n$,
\begin{align*}
d_\infty(S_{n+1}, \Lambda_{ \mu}) &\leq d_\infty(S_{n+1}, \hat{S}_{n+1}) + d_\infty(\hat{S}_{n+1}, \Lambda_{ \mu_k}) + d_\infty(\Lambda_{ \mu_k},\Lambda_{\mu}) \\
 &\leq {\varepsilon \over \alpha} + {\varepsilon \over \alpha} + {\varepsilon \over 2} \mbox{ a.s. }
\end{align*}
 Since $\displaystyle \varepsilon$\ is arbitrarily small, the proof is completed.
\end{proof}
 
\section{Applications to operator means}

Here, we follow the notations of the last section in \cite{lekapalfia}. The (non-commutative) perspectives of $f$ are given by the maps
$$x \mapsto f_x(y):=x^{1/2}f(x^{-1/2}yx^{-1/2})x^{1/2}.$$

By \cite[Proposition 2.8]{lekapalfia}, for any operator monotone function $f \colon (0, \infty) \rightarrow \mathbb{R}$ with $f(1)=0, f'(1)=1$, the perspectives $x\in\mathbb{P} \mapsto f_x(y):=x^{1/2}f(x^{-1/2}yx^{-1/2})x^{1/2}$ and their integrals
          $$x \mapsto \int_\mathbb{P}f_x(y) \: d\mu(y),$$
          where $\mu\in\mathcal{P}(\mathbb{P})$ such that
\begin{equation}\label{eq:intAssumption}
\int_\mathbb{P}\|f_x(y)\| \: d\mu(y)<\infty,
\end{equation}
          for any $x\in\mathbb{P},$
     define monotone (order-preserving) flow on $\mathbb{P}$ through the first-order system
\begin{equation}\label{eq:coreode}
      {\dot x}(t) = \int_\mathbb{P}f_{x(t)}(y) \: d\mu(y)
\end{equation}
  for $t\geq 0$. 
  The generalized Karcher equation
   \begin{align*}
      \int_\mathbb{P}f_{x}(y) \: d\mu(y) = 0
  \end{align*}
  has a unique solution, cf. \cite[Theorem 6.5]{lekapalfia}, denoted by $\Lambda_f(\mu)$. By Proposition 6.6 and the paragraph preceding Corollary 6.8 in \cite{lekapalfia},
  $\Lambda_f(\mu)$ is monotone increasing in its variables: with respect to the pointwise order for real functions $f$, and with respect to the stochastic ordering of probability measures $\mu$. This monotonicity property carries over to resolvents $J^\mu_\lambda(z)$ associated with the ODE \eqref{eq:coreode}.
  In fact, recall that the generalized resolvent $J^\mu_\lambda(z)$, $\lambda > 0,$ is the unique equilibrium of the ODE
   \begin{equation}\label{eq:coreode2} \dot{x} = \frac{\lambda}{1+\lambda} \int_\mathbb{P} f_x(y) \: d \mu(y) + \frac{1}{1+\lambda} \log_x z =: F_{z, \lambda}(x).
   \end{equation}
  From Proposition 2.8 in \cite{lekapalfia}, both dynamical systems
    $$ {\dot x}(t)  = \int_\mathbb{P} f_{x(t)}(y) \: d \mu(y) \; \text{ and } \;  {\dot x}(t) = \log_{x(t)} z $$
   generate an order-preserving flow on $\mathbb{P}$, so does their convex combination $F_{z, \lambda}$.
   Following the reasoning in Proposition 6.6 in \cite{lekapalfia}, we have that the order-interval $[J^\mu_\lambda(z), \infty)$ is an invariant set of the dynamical system \eqref{eq:coreode2}. Since $F_{z, \lambda} \leq F_{v, \lambda}$ holds when $z \leq v$, we also obtain that $0 = F_{z, \lambda}(J^\mu_\lambda(z)) \leq F_{v, \lambda}(J^\mu_\lambda(z))$ and then the flow generated by $F_{v, \lambda}$ leaves the set   $[J^\mu_\lambda(z), \infty)$ invariant, hence $J^\mu_\lambda(z) \leq J^\mu_\lambda(v)$ must follow.

     The order-preserving property of the resolvent in terms of generalized Karcher equations was first proved in \cite{palfia2} and also follows from the order-preserving properties of the associated flows (see \cite{lekapalfia}). The statement summarizing this is the following:
\begin{prop}\label{P:monotoneResolvent}
The generalized resolvent $J^\mu_\lambda(z)$ is monotone increasing with respect to the stochastic order of probability measures $\mu$ and the Loewner order on $z$.
\end{prop}   
   
The exponential metric increasing (EMI) property says that for any $x$ and $y \in \mathbb{P}$, the inequality holds
 $$ \|\log x - \log y\| \leq d_\infty(x,y),$$
which readily implies
 $$  \|\log x^{-1/2}yx^{-1/2} - \log x^{-1}\| \leq \|\log y\|. $$
The following is a localized version of the statement for general operator-monotone functions.

\begin{prop}\label{P:concave_lip}
   Let $f \colon (0,\infty) \rightarrow \mathbb{R}$ be an operator monotone function such that $f(1) = 0.$ Then, for any $x$ and $y \in \mathbb{P},$ the inequality
    $$\|f(x^{-1/2}yx^{-1/2}) - f(x^{-1})\| \leq K_x \|f(y)\|$$
   holds with $K_x = \max \left\{ \|x^{-1}\|,\|x\| \right\}.$
\end{prop}

\begin{proof}
   First, we prove the statement if $y$ is a positive scalar; that is, $y = \lambda > 0.$
   Assume that $1 \leq \lambda$ holds. From the integral representation of operator monotone functions
   $$0 \leq f(\lambda x^{-1}) - f(x^{-1}) = \int_{[0,\infty]} -{x \over \lambda + sx} + {x \over 1 + sx} \: d\nu(s) = \int_{[0,\infty]} {x(\lambda-1) \over (\lambda + sx)( 1 + sx)} \: d\nu(s), $$
   where $\nu$ is a positive Borel measure on $[0,\infty]$. If $x \in \mathbb{P}$ satisfies $1 \leq x$, we obtain
    $$  \int_{[0,\infty]} {x(\lambda-1) \over (\lambda + sx)( 1 + sx)} \: d\nu(s) \leq \int_{[0,\infty]} {x(\lambda-1) \over (\lambda + s)( 1 + s)} \: d\nu(s) = x(f(\lambda) - f(1)). $$
Similarly, we obtain that for $0 < \lambda \leq 1 $ and $1 \leq x \in \mathbb{P}$
 $$0 \leq f(x^{-1}) - f(\lambda x^{-1}) \leq x(f(1)-f(\lambda)),$$
hence, the norm inequality follows in this case.

If $1 \geq x$ and $1 \in \mathbb{P}$, we recall that $g(x) = -f(1/x)$ is also operator monotone. From the inequalities above for $g$, we obtain that
  $$\|f(x^{-1}) - f(\lambda x^{-1})\| = \|g(x) - g(x/\lambda)\| \leq \|x^{-1}\| |g(1/\lambda)| = \|x^{-1}\| |f(\lambda)|.$$

For a general $x \in \mathbb{P},$ let $p$ denote the spectral projection of $x$ corresponding to the Borel set $\sigma(x) \cap [1,\infty).$ With the notation $x_1 = px$ and $x_2 = (1-p)x$, we have
$f(x) = f(x_1) \oplus f(x_2).$
Then the preceding argument shows that
$$\|f(\lambda x^{-1}) - f(x^{-1}) \| = \max_i \|f(\lambda x_i^{-1}) - f(x_i^{-1}) \|  \leq \max (\|x^{-1}\| ,\|x\|)  |f(\lambda)|.$$

Finally, let $[\lambda_a, \lambda_b]$ denote the spectrum of any fixed $y \in \mathbb{P}.$   From the spectral mapping theorem, $\lambda_a \leq y \leq \lambda_b$ and $\|f(y)\| = \max(|f(\lambda_a)|,|f(\lambda_b)| ).$ Since $f$ is operator monotone, we also have
\begin{equation*}
\begin{split}
f(x^{-1/2}yx^{-1/2}) - f(x^{-1}) &\leq f(\lambda_b x^{-1}) - f(x^{-1}) \leq \max(\|x^{-1}\|, \|x\|)|f(\lambda_b)|\\
&\leq \max(\|x^{-1}\|, \|x\|)\|f(y)\|.
\end{split}
\end{equation*}
Similarly,
\begin{equation*}
\begin{split}f(x^{-1}) - f(x^{-1/2}yx^{-1/2})  &\leq  f(x^{-1}) - f(\lambda_a x^{-1}) \leq \max(\|x^{-1}\|, \|x\|)|f(\lambda_a)| \\
&\leq \max(\|x^{-1}\|, \|x\|)\|f(y)\|.
\end{split}
\end{equation*}
This completes the proof.
\end{proof}

\begin{cor}\label{C:UniformInt}
    Let $f \colon (0,\infty) \rightarrow \mathbb{R}$ be an operator monotone function such that $f(1) = 0$ and  $$\int_\mathbb{P} \|f(y)\| \: d\mu(y) < \infty. $$
    Then, for any $x \in \mathbb{P}$ and $\varepsilon > 0,$ the family of functions $$\mathcal{F}_{x,\varepsilon} = \left\{  \|f(z^{-1/2}yz^{-1/2})\| \colon d_\infty(x,z) < \varepsilon \right\}$$ admits a common $L^1$-dominating function.
\end{cor}

\begin{proof}
    The pointwise supremum $\sup \mathcal{F}_{x,\varepsilon}$ is lower semi-continuous, since each function is continuous in $\mathcal{F}_{x,\varepsilon}$; hence $\sup \mathcal{F}_{x,\varepsilon}$ is Borel measurable.
    From Proposition~\ref{P:concave_lip}, we obtain
     \begin{align*}
        \int_\mathbb{P}& \sup \mathcal{F}_{x,\varepsilon} \: d\mu  \leq \sup_{d_\infty(x,z) < \varepsilon} K_z \int_\mathbb{P} \|f(y)\| \: d\mu(y) + \sup_{d_\infty(x,z) < \varepsilon} \|f(z^{-1})\|.
     \end{align*}
      Since $$ \int_\mathbb{P} \|f(y)\| \, d\mu(y) < \infty $$ by assumption, and since both $K_z = \max \{\|z\|, \|z^{-1}\|\}$ and $ \|f(z^{-1})\| $ are bounded as $z $ ranges over an open neighborhood of $ x$, the above expression is finite. Thus, $\sup \mathcal{F}_{x,\varepsilon} \in L^1(\mu)$, and it admits the $L^1$-domination of $\mathcal{F}_{x, \varepsilon}$.
\end{proof}

\begin{cor}\label{C:LlogLInt}
    Let $f \colon (0,\infty) \rightarrow \mathbb{R}$ be an operator monotone function such that $f(1) = 0$ and  $$\int_\mathbb{P} \|f(y)\| \log^+ \|f(y)\| \: d\mu(y) < \infty. $$
    For any $x \in \mathbb{R}_+$, $$ f\left({\cdot\over x}\right)  \in L \log^+ L(\mathbb{R}_+,  \|\cdot\|_{\#}\mu).$$
   \end{cor}

\begin{proof}
From Corollary~\ref{C:UniformInt}, the functions $$y \mapsto \left\|f\left({y \over x}\right) \right\| \; \text{ and } \; y \mapsto \log^+ \left\|f\left({y \over x}\right) \right\|$$ belong to $L^1(\mu)$. Using Proposition~\ref{P:concave_lip},
$$ \left\|f\left({y \over x}\right) \right\| \log^+ \left\|f\left({y \over x}\right) \right\| \leq \left(K_x \|f(y)\| + \left|f\left(\frac1x\right)\right| \right) \log^+ \left(K_x \|f(y)\| + \left|f\left(\frac1x\right)\right|\right).$$
Recall that, for any $a$ and $b > 0$,
  $$
\log^+(a+b)
\leq \log 2+\max\{\log^+a,\log^+b\}
\leq \log 2+\log^+a+\log^+b.
 $$
 Thus, we conclude that $\left\|f\left(\frac{\cdot}{x}\right) \right\| \in L \log^+ L(\mathbb{P}, \mu)$.

  The spectral mapping theorem gives that $f\left(\frac{\|y\|}{x}\right) \in \sigma\left(f(\frac{y}{x})\right)$, and then $$\left|f\left(\frac{\|y\|}{x}\right)\right| \leq \left\|\frac{f(y)}{x}\right\|.$$ Since the function $t \mapsto f(t) \log^+ f(t)$ is nondecreasing on $(0, \infty)$, it follows $$f\left(\frac{\cdot}{x} \right)  \in L \log^+ L(\mathbb{R}_+,  \|\cdot\|_{\#}\mu).$$

\end{proof}

In order to obtain a.s. boundedness of our resolvent iterates, in what follows we postulate a slightly weaker integrability condition than the $L^1$ condition of \eqref{eq:intAssumption}. It lies between an $L^1$ and $L^2$ condition and we believe that merely \eqref{eq:intAssumption} should be enough for a.s. boundedness, but we were not able to justify this. We believe that a more refined analysis might be required for this which takes into account the varying exponential contraction coefficients along the trajectories of the resolvents iterates. For the following proofs mere nonexpansivity is enough, but then the random harmonic series type error terms, as utilized below, grow too quickly for \eqref{eq:intAssumption} to be effective. The postulated additional integrability of $x\log^{+}(x)$ below is still a significantly weaker assumption than \cite[(26)]{lekapalfia} which essentially requires the existence of all moments of $d_\infty(I,Y_1)$ (equivalently $\mathbb{E}|Y_1|<\infty$) that has not much in common with $\|f(y)\|$, and a.s. lower and upper bounds required seemingly different integrability assumptions. In what follows we also unify the picture regarding the establishment of claimed a.s. lower and upper bounds. The key tool to obtain boundedness is the following a.s. convergence result.

\begin{prop}\label{P:ConvergentStochastic}
   Let $X_1, X_2, ... $ be a sequence of i.i.d. real-valued random variables, and let $h \colon \mathbb{R} \to \mathbb{R}$ be a non-decreasing function such that
 $$ \mathbb{E}[|h(X_1 - r)|\log^+ |h(X_1 - r)|] < \infty$$ for every real $r$.
Assume that the equation $$\mathbb{E}h(X_1 - x) = 0$$ has the unique solution $x=0$.
 For arbitrary initial values $s_1,\hat{s}_1  \in \mathbb{R}$, consider the implicit stochastic approximations
$$\frac1n h(X_n - s_n) + s_{n-1} = s_n $$
and
$$
\frac1n h(X_n - \hat{s}_n) + h(\hat{s}_{n-1}-\hat{s}_n) = 0.
$$
The resulting sequences $\{s_n\}_{n \geq 1}$ and $\{\hat{s}\}_{n \geq 1}$ both converge to $0$ almost surely, given that in the case of the second iteration $\{\hat{s}\}_{n \geq 1}$, we additionally have that $h^{-1}$ exists and satisfies $$h^{-1}(x)=x+O(x^2)$$ in a neighborhood of $0$.
\end{prop}
\begin{proof}
The proof will utilize the claim that
  $$ \sum_{n=1}^\infty \frac{h(X_n)}{n}$$
  converges almost surely. We will apply Kolmogorov’s three-series theorem to establish the claim, it therefore suffices to verify the following three conditions. First,
  $$ \sum_{n=1}^\infty P(|h(X_n)| > n) =  \sum_{n=1}^\infty P(|h(X_1)| > n) \leq \mathbb{E}|h(X_1)| < \infty.$$
  Second, from $\mathbb{E}h(X_1)=0$ and Fubini's theorem,
\begin{equation*}
\begin{split}
\sum_{n = 1}^\infty \frac{\mathbb{E}\left(h(X_{n})1_{\{|h(X_{n})| < n\}}\right)}{n}&= \sum_{n = 1}^\infty \frac{\mathbb{E}\left(h(X_{1}) 1_{\{|h(X_{1})| < n\}}\right)}{n} \\
&= -\sum_{n = 1}^\infty \frac{\mathbb{E}\left(h(X_{1})1_{\{|h(X_{1})|\geq n\}}\right)}{n}\\
&= -\mathbb{E} \left( h(X_{1}) \sum_{n = 1}^\infty \frac{1_{\{|h(X_{1})|\geq n\}}}{n} \right)\\
&\leq O\left(\mathbb{E}\left[|h(X_1)|\log^{+}|h(X_1)|\right]\right)<\infty.
\end{split}
\end{equation*}
Third, a standard estimate from \cite[p. 264]{Dudley} yields
\begin{equation}\label{eq:P:ConvergentStochastic}
\begin{split}
\sum_{n = 1}^\infty \frac{1}{n^2}\mathbb{E}\left(h(X_{n})^21_{\{|h(X_{n})|<n\}}\right)
&\leq \sum_{n = 1}^\infty \frac{1}{n^2}\int_0^n x^2 \: dF_{|h(X_{n})|}\\
&\leq O\left(\sum_{n = 0}^\infty \int_n^{n+1} x^2 \: dF_{|h(X_{n})|} \frac{1}{n+1}\right)\\
&\leq O\left(\sum_{n =  0}^\infty \int_n^{n+1} x \: dF_{|h(X_{n})|} \right)\leq O\left(\mathbb{E}|h(X_{1})|\right)<\infty,
\end{split}
\end{equation}
where $F_{|h(X_{n})|}$ denotes the distribution function of $|h(X_{n})|$.

  Thus, using Kolmogorov’s three-series theorem \cite[Theorem 2.16]{HallHeyde}, we obtain that the series
  $$ \sum_{n = 1}^\infty \frac{h(X_n)}{n}$$
  converges almost surely.

  Let $\{X_n\}_{n=1}^\infty$ be a sequence appearing in the convergent harmonically weighted series. First, assume that the corresponding sequence $\{s_n\}_{n=1}^\infty$ changes sign infinitely often. Let $n_0$ be an index such that $s_{n_0-1} \geq 0$ and $s_{n_0},..., s_M < 0$. Summing the equations of the implicit recursion, we obtain
\begin{equation}\label{eq1:GPT}
s_M = s_{n_0-1} + \sum_{n=n_0}^M \frac{h(X_n)}{n} + \sum_{n=n_0}^M \frac1n (h(X_n - s_n) - h(X_n)).
\end{equation}
 Since $h$ is non-decreasing, the last sum is clearly nonnegative. From $s_M < 0$ and $s_{n_0-1} \geq 0$,
  $$ |s_M| \leq \left|\sum_{n=n_0}^M \frac{h(X_n)}{n}\right|$$
  must hold.
 We can assume that $n_0$ is arbitrarily large, and then the sum of the right-hand side is also uniformly arbitrarily small, which implies that $s_M$ is arbitrarily small for sufficiently large $M$. Similar reasoning applied to a block of positive terms in $\{s_n\}_{n=1}^\infty$ implies that the sequence $\{s_n\}_{n=1}^\infty$ must converge to $0$.

Next, assume that $\{s_n\}_{n=1}^\infty$ is positive for every index $n \geq n_0$. From the equality
     $$ s_M = s_{n_0-1} + \sum_{n=n_0}^M \frac{h(X_n)}{n} + \sum_{n=n_0}^M \frac1n (h(X_n - s_n) - h(X_n))$$
we obtain that $s_M$ must be a convergent sequence when $M \to \infty$. Indeed, the sequence $\displaystyle \sum_{n=n_0}^M \frac{h(X_n)}{n}$ is clearly convergent when $M \to \infty$. Furthermore, $\sum_{n=n_0}^M \frac1n (h(X_n - s_n) - h(X_n)$ is a bounded nonincreasing sequence because $s_n \geq 0$ and
    $$ \sum_{n=n_0}^M \frac1n (h(X_n - s_n) - h(X_n)) = s_M - s_{n_0-1} - \sum_{n=n_0}^M \frac{h(X_n)}{n} \geq - s_{n_0-1} - \sum_{n=n_0}^M \frac{h(X_n)}{n} > -\infty;$$ hence it is convergent. As a sum of convergent sequences, ${s_M}$ also converges; denote its limit by $L$. We prove that $L=0$ must hold. Let us assume that $L \neq 0$, for instance $L > 0$. Then, for every large index $n$, $L/2 \leq s_n$ holds. We also have
 $$\mathbb{E}h(X_n  - s_n) \leq \mathbb{E}h(X_n - L/2) =: \mu < \mathbb{E}h(X_n) = 0.$$
Then, summing the equations of recursion from $n_0$ to $M$, we obtain
 $$ \sum_{n=n_0}^M \frac1n h(X_n - s_n) = s_M - s_{n_0}.$$
Using Etemadi's weighted SLLN \cite[Theorem 1]{Etemadi},
 $$   \sum_{n=n_0}^M \frac1n h(X_n - s_n) \leq  \sum_{n=n_0}^M \frac1n h(X_n - L/2) \sim \log (M-n_0) \mu;$$
hence,
 $$ s_M - s_{n_0} = \sum_{n=n_0}^M \frac1n h(X_n - s_n) \lesssim  \log(M-n_0) \mu \to -\infty, \quad M \to \infty.$$
 Since left-hand side is bounded, our assumption leads to a contradiction, hence $L=0$ must hold. Similarly, if $\{s_n\}_{n=1}^\infty$ is negative for every index $n \geq n_0$, we obtain that same conclusion. This completes the proof for the sequence $\{s_n\}_{n \geq 1}$.
 
 For the second iteration $\hat{s}_n$, since $h$ is increasing, so is $h^{-1}$. Therefore,
\begin{equation*}
 -h^{-1}\left(-\frac1n h(X_n - \hat{s}_n)\right) + \hat{s}_{n-1}=\hat{s}_n
\end{equation*} 
and $-h^{-1}\left(-\frac1n h(x)\right)$ is also increasing. Summing the equations from $n_0$ to $M$, we similarly obtain
\begin{equation}\label{eq2:GPT}
\begin{split}
\hat{s}_M &= s_{n_0-1} + \sum_{n=n_0}^M -h^{-1}\left(-\frac{h(X_n)}{n}\right)\\
 &\quad + \sum_{n=n_0}^M -h^{-1}\left(-\frac1n (h(X_n - \hat{s}_n)\right) + h^{-1}\left(-\frac1n h(X_n))\right).
\end{split}
\end{equation}
The Borel--Cantelli lemma implies that $\{|X_n| > n\}$ occurs finitely many time. Hence,
the inequality \eqref{eq:P:ConvergentStochastic} yields
 $$\sum_{n=1}^\infty\frac{h(X_n)^2}{n^2} $$ is convergent a.s..
Using $h^{-1}(x)=x+O(x^2)$ in a neighborhood of $0$, we have
\begin{equation*}
\sum_{n=n_0}^M -h^{-1}\left(-\frac{h(X_n)}{n}\right)=\sum_{n=n_0}^M \frac{h(X_n)}{n}+O\left(\frac{h(X_n)^2}{n^2}\right),
\end{equation*}
that is convergent a.s. when $n_0$ is sufficiently large and $M \to \infty$. Thus, using \eqref{eq2:GPT} in place of \eqref{eq1:GPT} we can follow the proof above for $s_n$ to obtain that $\hat{s}_n$ converges to some $L\in\mathbb{R}$ a.s.. Similarly, assuming first that $0 < L$, we have that
\begin{equation*}
\begin{split}
\hat{s}_M - \hat{s}_{n_0} &= \sum_{n=n_0}^M -h^{-1}\left(-\frac1n h(X_n - \hat{s}_n)\right)\leq \sum_{n=n_0}^M -h^{-1}\left(-\frac1n h(X_n - L/2)\right)\\
&\leq \sum_{n=n_0}^M \frac{h(X_n-L/2)}{n}+O\left(\frac{h(X_n-L/2)^2}{n^2}\right)\\
&\lesssim  \log(M-n_0) \mathbb{E}h(X_n-L/2)+\sum_{n=n_0}^M O\left(\frac{h(X_n-L/2)^2}{n^2}\right) \to -\infty, \quad M \to \infty
\end{split}
\end{equation*}
a.s., because, as established earlier, $\mathbb{E}h(X_n-L/2)<0$ and $\sum_{n\geq n_0}\frac{1}{n^2}h(X_{n}-r)^2<\infty$ a.s.. Then, the same arguments establishing that $L=0$ for $s_n$ a.s. implies that $L=0$ a.s. for $\hat{s}_n$ as well. The proof is complete.
\end{proof}

\begin{lemma}\label{L:bndd2}
 Let $f$ be an operator monotone function on $(0, \infty)$ such that $f(1) = 0$ and $f'(1) = 1$.
 Let $Y_1, Y_2, \hdots$  be a sequence of i.i.d. random variables with probability law $\mu$ on $\mathbb{P}$ such that
     $$ \int_\mathbb{P} \|f(y)\|\log^{+}\|f(y)\| \: d\mu(y) < \infty.$$ 
     For any initial $S_1\in\mathbb{P}$, the stochastic resolvent iteration $$S_{n+1}=J_{1 \over n+1}(Y_{n+1},S_n) \qquad n \geq 1$$ is almost surely bounded.
\end{lemma}
\begin{proof}
  To prove that the iteration
    \begin{equation}\label{eq1:L:bndd2}
   {1 \over n+1} f_{S_{n+1}}(Y_{n+1}) + \log_{S_{n+1}}(S_n) = 0 \qquad n \geq 1,
   \end{equation}
   generates an almost surely bounded sequence, we reduce the problem to the stability of an implicit stochastic approximation on the real line.
   
   Let us consider the following implicit stochastic approximation algorithm defined by $s_1:=\|S_1\|$ and
    \begin{equation}\label{eq3:L:bndd2}
    {1 \over n+1} f_{s_{n+1}}(\|Y_{n+1}\|) + \log_{s_{n+1}}\{s_n\} = 0 \qquad n \geq 1.
   \end{equation}
   Since $Y_{n+1}\leq \|Y_{n+1}\|I$ and $S_1\leq s_1I=\|S_1\|I$, a repeated application of the order-preserving property of the resolvent map in Proposition~\ref{P:monotoneResolvent} implies
\begin{equation*}
S_n\leq s_nI \qquad n \geq 1
\end{equation*}
   holds a.s.; thus, to find an upper bound on $S_n$, it is enough to prove that $s_n$ is a.s. bounded. Observe that, under the change of variables $s_n=e^{c_n}$ and $\|Y_{n}\|=e^{r_n}$, the recursion \eqref{eq3:L:bndd2} is equivalent to
   \begin{equation*}
{1 \over n+1} h(r_{n+1}-c_{n+1}) + c_n - c_{n+1}= 0 \qquad n \geq 1,
   \end{equation*}
   where $h = f \; \circ \; \exp$ increases on the real line.
 Furhtermore, from Corollary~\ref{C:LlogLInt}, we obtain that $h(r_1 - r) \in L\log^+ L(\mathbb{R},(\log \|\cdot\|)_\# \mu)$ for every real $r$ and that the transformed equation
  $$\mathbb{E}h(r_1 - r) = 0$$ has a unique solution $r\in\mathbb{R}$, since the original equation
 \begin{align*}
      0=\mathbb{E}f_{x}(\|Y_1\|)
  \end{align*}
  has a unique solution $x>0$, cf. \cite[Theorem 6.5]{lekapalfia}.
 Using Proposition~\ref{P:ConvergentStochastic}, the sequence $\{s_n\}_{n \geq 1}$ is a.s. convergent. Hence, $S_n$ has a finite upper bound a.s..

To get a lower bound, by the similarity equivariance of the holomorphic functional calculus, namely, $f(SAS^{-1})=Sf(A)S^{-1},$
the recursion \eqref{eq1:L:bndd2} can be rewritten as
    \begin{equation*}
    {1 \over n+1} g_{S_{n+1}^{-1}}(Y_{n+1}^{-1}) + \log_{S_{n+1}^{-1}}(S_n^{-1}) = 0 \qquad n \geq 1,
   \end{equation*}
   where $g(x) = -f(1/x)$ is also an operator monotone function, and $g(1)=0$ and $g'(1)=1$ also hold.
Exactly the same analysis leads to the bound
\begin{equation*}
S_n^{-1}\leq s_nI \qquad n \geq 1,
\end{equation*}
   where the stochastic sequence of reals defined as $s_1:=\|S_1^{-1}\|$ and
    \begin{equation*}
    {1 \over n+1} g_{s_{n+1}}(\|Y_{n+1}^{-1}\|) + \log_{s_{n+1}}\{s_n\} = 0 \qquad n \geq 1.
   \end{equation*}
Indeed, from the spectral mapping theorem, $\|g(Y_{1}^{-1})\|\in L \log^+ L(\mathbb{P}, \mu)$ also holds. Hence, by the same argument as above, with $g$ in place of $f$, we obtain an upper bound for $S_n^{-1}$, thus a lower bound for $S_n$ a.s., the proof is concluded.
\end{proof}

We consider here, just as in the last section of \cite{lekapalfia}, the "symmetrized" resolvent $M^f_t(a,b):=\Lambda_f((1-t)\delta_a+t\delta_b)$ for $t\in[0,1]$ and $a,b\in\mathbb{P}$ which is defined as the unique solution $x\in\mathbb{P}$ of the generalized Karcher equation
  \begin{equation*}
    tf_x(b) + (1-t)f_{x}(a) = 0,
  \end{equation*}
cf. \cite[Theorem 6.5]{lekapalfia}.
The following result provides a relaxed variant of the a.s. boundedness established in the first part of the proof of \cite[Theorem 6.9]{lekapalfia} which was based on the operator concavity of $f$ and the existence of $\mathbb{E}\|Y_1\|$ and $\mathbb{E}\|Y_1^{-1}\|$.

 \begin{lemma}\label{L:bndd2b} Let $f$ be an operator monotone function on $(0, \infty)$ such that $f(1) = 0$ and $f'(1) = 1$.
     Let $Y_1, Y_2, \hdots$  be a sequence of i.i.d. random variables with probability law $\mu$ on $\mathbb{P}$ such that
     $$ \int_\mathbb{P} \|f(y)\|\log^{+}(\|f(y)\|) \: d\mu(y) < \infty.$$
    For any initial $S_1\in\mathbb{P}$, the stochastic iteration $S_{n+1}=M^f_{1 \over n+2}(Y_{n+1}, S_n)$  is almost surely bounded.
  \end{lemma}
  \begin{proof}
  We will utilize the proof of the previous Lemma~\ref{L:bndd2} in order to find upper and lower bound for the iterates. Notice that $M^f_{1 \over n+2}(Y_{n+1}, S_n)$ is the unique solution of the modified resolvent equation (itself a generalized Karcher equation)
  \begin{equation*}
    \frac{1}{n+2} f_{{S}_{n+1}}(Y_{n+1}) + \frac{n+1}{n+2} f_{{S}_{n+1}}({S}_n) = 0,
  \end{equation*}
  thus, just as in Proposition~\ref{P:monotoneResolvent}, is operator monotone by \cite[Proposition 6.6]{lekapalfia} in both of its variables $(Y_{n+1}, S_n)$.  Given the sequence of i.i.d. random variables $\|Y_1\|, \|Y_2\|, \hdots$, we consider the stochastic sequence $\{s_n\}_{n \geq1}$ of real numbers defined as $s_1:=\|{S}_1\|$ and
   \begin{equation}\label{eq5:L:bndd2b}
    {1 \over n+1} f_{s_{n+1}}(\|Y_{n+1}\|) + f_{s_{n+1}}\{s_n\} = 0 \qquad n \geq 1.
   \end{equation}
   Since $Y_n \leq \|Y_n\|I$ and $S_1 \leq s_1 I$, repeated application of the order-preserving property of the resolvent
   implies
   \begin{equation*}
{S}_n\leq s_nI.
\end{equation*}   
Thus, we can follow the proof of Lemma~\ref{L:bndd2} and transform the problem to $\mathbb{R}$ by an exponential substitution. Indeed, applying the second statement in Proposition~\ref{P:ConvergentStochastic} to $h = f \circ \exp$, we obtain that the generated sequence is convergent; i.e. $s_n \to r$ a.s. if $n \to \infty$, where $r\in\mathbb{R}$ is the unique solution to the equation $\mathbb{E}f_{x}(\|Y_1\|)=0$. Hence, $S_n$ has a finite upper bound almost surely.

Now following the the last paragraph of the proof of Lemma~\ref{L:bndd2} we have the following steps which follow one after another:
    \begin{align*}
    {1 \over n+1} {S}_{n+1}f({S}_{n+1}^{-1}Y_{n+1}) + {S}_{n+1}f({S}_{n+1}^{-1}{S}_n) &= 0 \qquad n \geq 1,\\
    {1 \over n+1} f(Y_{n+1}{S}_{n+1}^{-1}){S}_{n+1} + f({S}_n{S}_{n+1}^{-1}){S}_{n+1} &= 0 \qquad n \geq 1,\\
    {1 \over n+1} g({S}_{n+1}Y_{n+1}^{-1}) + g({S}_{n+1}{S}_n^{-1}) &= 0 \qquad n \geq 1,\\
    {1 \over n+1} g_{{S}_{n+1}^{-1}}(Y_{n+1}^{-1}) + g_{{S}_{n+1}^{-1}}({S}_n^{-1}) &= 0 \qquad n \geq 1.
   \end{align*}
   Now monotonicity of ${S}_{n+1}^{-1}$ in $(Y_{n+1}^{-1},{S}_{n}^{-1})$ implied by the last equation above, combined with a stochastic sequence of real numbers defined as $s_1^{-1}:=\|S_1^{-1}\|$ and
    \begin{equation*}
    {1 \over n+1} g_{s_{n+1}^{-1}}(\|Y_{n+1}^{-1}\|) + g_{s_{n+1}^{-1}}(s_n^{-1}) = 0 \qquad n \geq 1,
   \end{equation*}
   yields again
   \begin{equation*}
{S}_n^{-1}\leq s_n^{-1}I.
\end{equation*}
  From here, we follow again the last paragraph of the proof of Lemma~\ref{L:bndd2} to obtain that this generated sequence $\{{s}_n^{-1}\}_{n\in\mathbb{N}}$ is a.s. bounded from above.
\end{proof}

Given the above lemmas providing almost sure boundedness of the stochastic resolvent sequences, the proof of the following strong law essentially follows the second part of the proof of \cite[Theorem 6.9]{lekapalfia} that we shall expand here further for the sake of clarity.

 \begin{thm}[Strong Law of Large Numbers]\label{T:slln2}
    Let $f$ be an operator monotone function on $(0, \infty)$ such that $f(1) = 0$ and $f'(1) = 1.$
    Let $Y_1, Y_2, \hdots$  be a sequence of i.i.d. random variables with probability law $\mu$ on $\mathbb{P}$ such that
     $$ \int_\mathbb{P} \|f(y)\|\log^{+}(\|f(y)\|) \: d\mu(y) < \infty.$$
     For any initial $\hat{S}_1\in\mathbb{P}$, the stochastic resolvent iteration $\hat{S}_{n+1}=J_{1 \over n+1}(Y_{n+1},\hat{S}_n)$ converges almost surely to $\Lambda_f(\mu)$ with respect to the Thompson metric.
     
     Similarly, for any initial $S_1\in\mathbb{P}$, the stochastic iteration $$S_{n+1}=M^f_{1 \over n+2}(Y_{n+1}, S_n)$$ converge almost surely to $\Lambda_f(\mu)$ with respect to the Thompson metric.
  \end{thm}
  \begin{proof}
     Let us recall that by \cite[Proposition 6.4]{lekapalfia} the first-order system
     $$ \dot{x} = \int_\mathbb{P} f_x(y) \: d\mu(y)$$
    leaves invariant any sufficiently large order-interval $Q = [c^{-1}I,cI].$
     From Lemma~\ref{L:bndd2} stochastic resolvent iteration $\hat{S}_n$ is bounded in $d_\infty$ a.s.; thus, there exists a $c>0$, depending on $\omega\in\Omega$, such that $\hat{S}_n \in [c^{-1}I, cI]$.
    We note that the exponential contraction coefficient $\alpha(Q,f_x(y))$ of the $Q$-invariant flow is strictly positive for any $y \in \mathbb{P}$ (see the proof of \cite[Theorem 6.1]{lekapalfia}. Moreover, $\mathbb{E}(\alpha(Q, f_x(Y)) < \infty$ because without loss of generality we can assume that $\alpha(Q,f_x(y))\leq 1$. Hence, the SLLN and \cite[Theorem 2.1]{lekapalfia} implies
    $${1 \over k} \sum_{i=1}^k \alpha(f_x(Y_i)) \rightarrow \mathbb{E}(\alpha(Q, f_x(Y))) > 0.$$
    We can now apply the proof of Theorem~\ref{T:mainStochasticConvergence}
to the invariant set $Q$, using the local $L^1$-domination condition
from Corollary~\ref{C:UniformInt}, to obtain
$$
d_\infty(\hat{S}_n,\Lambda_f(\mu)) \rightarrow 0
\qquad \text{a.s.}.
$$

   To prove the convergence of the symmetrized iteration, we claim that $d_\infty(\hat{S}_{n},{S}_n)\to 0$ a.s.
   From Lemma~\ref{L:bndd2} and Lemma~\ref{L:bndd2b}, both stochastic resolvent sequences ${S}_n$ and $\hat{S}_n$ are bounded in $d_\infty$ a.s.; thus, there exists a large enough $c>0$, depending on $\omega\in\Omega$, such that for the order-interval $Q = [c^{-1}I,cI]$ we have ${S}_n,\hat{S}_n\in Q$ a.s. for each $n\in\mathbb{N}$. From the Taylor series expansions, we have
      $$ f(x) = x-1 + O((x-1)^2) \quad \mbox{ and } \quad \log x = x-1 + O((x-1)^2)$$ on bounded sets.
     Clearly, $J_{1 \over n+1}(Y_{n+1},{S}_n)$ is the only fixed point of the differential equation
\begin{equation}\label{eq:ode}
   \dot{x}(t) = {1 \over n+1} f_x(Y_{n+1}) + \log_{x}S_n.
\end{equation}
By \cite[Corollary 2.2]{lekapalfia}, recall that for all $n\geq n_0(\omega)$ $$\alpha \left({1 \over n+1} f_x(Y_{n+1}) + \log_{x}S_n \right) \geq 1+\frac{\alpha(Q, f_x(Y_{n+1}))}{n+1}.$$
 Note that
  \begin{equation}\label{eq1:T:slln2}
  \begin{split}
     0 &= {1 \over n+1} f_{{S}_{n+1}}(Y_{n+1}) + f_{{S}_{n+1}}({S}_n)\\
     &= \left({1 \over n+1} f_{{S}_{n+1}}(Y_{n+1}) + \log_{{S}_{n+1}}{S}_{n}\right) + \left(f_{{S}_{n+1}}({S}_n)-\log_{{S}_{n+1}}{S}_{n}\right).
  \end{split}
  \end{equation}
So the contraction property \cite[Theorem 2.4]{lekapalfia} applied to the ODE \eqref{eq:ode} and \eqref{eq1:T:slln2} imply that
  \begin{equation*}
  \begin{split}
  &d_\infty(J_{1 \over n+1}(Y_{n+1},{S}_n),{S}_{n+1})=d_\infty(J_{1 \over n+1}(Y_{n+1},{S}_n),M^f_{1 \over n+2}(Y_{n+1}, S_n))\\
  &\leq \left(1+\alpha\left(Q, \frac{f_x(Y_{n+1})}{n+1}\right)\right)^{-1}\left\|{S}_{n+1}^{-1/2}\left({1 \over n+1} f_{{S}_{n+1}}(Y_{n+1}) + \log_{{S}_{n+1}}({S}_{n})\right){S}_{n+1}^{-1/2}\right\|\\
  &=\left(1+\frac{\alpha\left(Q, f_x(Y_{n+1})\right)}{n+1}\right)^{-1}\|{S}_{n+1}^{-1/2}(f_{{S}_{n+1}}({S}_n)-\log_{{S}_{n+1}}({S}_{n})){S}_{n+1}^{-1/2}\|\\
  &=\left(1+\frac{\alpha\left(Q, f_x(Y_{n+1})\right)}{n+1}\right)^{-1}\|f({S}_{n+1}^{-1/2}{S}_n{S}_{n+1}^{-1/2})-\log({S}_{n+1}^{-1/2}{S}_{n}{S}_{n+1}^{-1/2})\|\\
  &\leq \left(1+\frac{\alpha\left(Q, f_x(Y_{n+1})\right)}{n+1}\right)^{-1}O\left(\left\|{S}_{n+1}^{-1/2}{S}_n{S}_{n+1}^{-1/2}-I\right\|^2\right)
  \end{split}
  \end{equation*} 
   where to get the last inequality, we used the fact that the power series expansions of $f$ and $\log$ about $I$ agree up to first order. Since $S_n$ is a.s. bounded, we have for all $n\geq n_0(\omega)$ that
   $$ O\left(\left\|{S}_{n+1}^{-1/2}{S}_n{S}_{n+1}^{-1/2}-I\right\|\right) = O(\|f(S^{-1/2}_{n+1} S_{n} S^{-1/2}_{n+1})\|)={1 \over n+1 }O(\|f({S}_{n+1}^{-1/2}Y_{n+1}{S}_{n+1}^{-1/2}) \|).$$
   Using Proposition~\ref{P:concave_lip}, this implies
   \begin{equation*}
   O\left(\left\|{S}_{n+1}^{-1/2}{S}_n{S}_{n+1}^{-1/2}-I\right\|\right) \leq {1 \over n+1 }O\left(C_1\|f(Y_{n+1})\|+C_2\right).
   \end{equation*}
   Now, we estimate
   \begin{equation}\label{eq3:T:slln2}
   \begin{split}
   &d_\infty(\hat{S}_{n+1},{S}_{n+1})\\
   &\leq d_\infty\left(J_{1 \over n+1}(Y_{n+1},\hat{S}_n),J_{1 \over n+1}(Y_{n+1},{S}_n)\right)+d_\infty\left(J_{1 \over n+1}(Y_{n+1},{S}_n),{S}_{n+1}\right)\\
   &\leq \left(1+\frac{\alpha\left(Q, f_x(Y_{n+1})\right)}{n+1}\right)^{-1}\left[d_\infty(\hat{S}_{n},{S}_n)+O\left(\left\|{S}_{n+1}^{-1/2}{S}_n{S}_{n+1}^{-1/2}-I\right\|^2\right)\right]\\
   &\leq \left(1+\frac{\alpha\left(Q, f_x(Y_{n+1})\right)}{n+1}\right)^{-1}\left[d_\infty(\hat{S}_{n},{S}_n)+{1 \over (n+1)^2 }O\left((C_1\|f({Y}_{n+1})\|+C_2)^2\right)\right].
   \end{split}
   \end{equation}
   Note that the product of the terms in front of the last expression satisfies
   \begin{equation}\label{eq4:T:slln2}
   \prod_{n\geq 0}\left(1+\frac{\alpha\left(Q, f_x(Y_{n+1})\right)}{n+1}\right)^{-1}=0\quad\text{ a.s.,}
   \end{equation}
   otherwise using the SLLN would contradict the fact that $\mathbb{E}\left[\alpha\left(B, f_x(Y_{n+1})\right)|\mathfrak{F}_n\right]>0$ for any fixed large enough bounded metric ball $B\subseteq\mathbb{P}$ intersecting $\supp(\mu)$.
   Now, an upper estimate of sums of expectations of truncations as in \eqref{eq:P:ConvergentStochastic} and the Borel-Cantelli lemma imply again that
   \begin{equation*}
   \sum_{n\geq 0}{1 \over (n+1)^2 }O\left((C_1\|f({Y}_{n+1})\|+C_2)^2\right)<\infty \quad\text{ a.s.}
   \end{equation*}
   as well, in particular the tail
   \begin{equation*}
   \sum_{n\geq m}{1 \over (n+1)^2 }O\left((C_1\|f({Y}_{n+1})\|+C_2)^2\right)
   \end{equation*}   
   a.s. goes to $0$ as $m\to\infty$. Since $d_\infty(\hat{S}_{n},{S}_n)$ is bounded a.s., this and then explicitly solving the recursion given by the initial state $d_{m}:=d_\infty(\hat{S}_{m},{S}_{m})$ and for $n\geq m$
   \begin{equation*}
   d_{n+1}:=\left(1+\frac{\alpha\left(Q, f_x(Y_{n+1})\right)}{n+1}\right)^{-1}\left[d_n+{1 \over (n+1)^2 }O\left((C_1\|f({Y}_{n+1})\|+C_2)^2\right)\right]
   \end{equation*}   
   combined with \eqref{eq3:T:slln2},\eqref{eq4:T:slln2} imply that $d_\infty(\hat{S}_{n},{S}_n)\leq d_n  < \varepsilon$ a.s. for all large enough $n\geq m\in\mathbb{N}$ for arbitrary $\varepsilon > 0$ proving our claim. The proof is complete.

   \end{proof}

\section*{Acknowledgments}
The authors acknowledge the help of ChatGPT Sol 5.6 in obtaining simplified and somewhat extended arguments within the proof of Proposition~\ref{P:ConvergentStochastic}.

The work of M.~P\'alfia was supported by the Ministry of Innovation and Technology of Hungary from the National Research, Development and Innovation Fund and financed under the TKP2021-NVA funding scheme, Project no. TKP2021-NVA-09; the Hungarian Scientific Research fund NKFIH ADVANCED-150059 and by the János Bolyai Research Scholarship of the Hungarian Academy of Sciences, Grant No. BO/00998/23/3.

\end{document}